%% file: master.tex
\documentclass[preprint,12pt]{elsarticle}

\include{packages}

\begin{document}

\begin{frontmatter}
	\title{ Cyclotomic matrices over real quadratic integer rings } 
	\author{ Gary Greaves }
	\address{Mathematics Department, Royal Holloway, Egham, Surrey, TW20 0EX, UK.}
	\ead{g.greaves@rhul.ac.uk}

	\begin{abstract}
		We classify all cyclotomic matrices over real quadratic integer rings and we show that this classification is the same as classifying cyclotomic matrices over the compositum all real quadratic integer rings, $\mathcal R$.
		Moreover, we enumerate a related class of symmetric $\mathcal R$-matrices; those $\mathcal R$-matrices whose eigenvalues are contained inside the interval $[-2,2]$ but whose characteristic polynomials are not in $\Z[x]$.
	\end{abstract}

\end{frontmatter}


\section{Introduction}
\label{sec:intro}
Let $A$ be an $n \times n$ Hermitian matrix with its characteristic polynomial $\chi_A(x) = \det(xI - A)$ having integer coefficients.
If $A$ also has all its eigenvalues in the interval $[-2,2]$ then we call it a \textbf{cyclotomic} matrix.
Cyclotomic matrices were first studied explicitly by McKee and Smyth who classified all cyclotomic matrices over the integers~\cite{McKee:IntSymCyc07}.
They are so named since, by a theorem of Kronecker~\cite{Kron:cyclo57}, for a cyclotomic matrix $A$, the polynomial $z^n\chi_A(z+1/z)$ is the product of some cyclotomic polynomials.
Over imaginary quadratic integer rings, cyclotomic matrices have been classified \cite{Greaves:CycloEG11, GTay:cyclos10}.
In this paper, we use the methods from the author's earlier paper \cite{Greaves:CycloEG11} to classify cyclotomic matrices over real quadratic integer rings, thereby completing the classification of cyclotomic matrices over quadratic integer rings.
We focus on the complication one encounters when studying cyclotomic matrices over the real quadratic integers, namely, the question of the integrality of the characteristic polynomial.
We also take advantage of a feature of the real quadratic integers, i.e., the ordering, and classify cyclotomic matrices over the compositum of all real quadratic integer rings.

Part of the reason for the study of cyclotomic matrices stems from a conjecture of Lehmer \cite{Lehmer:33Cyclo}.
Let $f(x) = (x - \alpha_1)\dots (x - \alpha_n)$ be a monic polynomial with integer coefficients.
Its \textbf{Mahler measure} \cite{Mahler:Measure1962} is defined as
\[
	M(f) = \prod_{j = 1}^n \max(1, \abs{\alpha_j}).
\]
Lehmer's problem is one of finding a monic integer polynomial $f$ with smallest possible Mahler measure $M(f)$ such that $M(f) > 1$.
The polynomial
\[
	L(z) = z^{10} + z^9 -z^7 - z^6 - z^5 - z^4 - z^3 + z + 1,
\]
whose larger real zero is $\Omega = 1.176280818\dots$, is known as Lehmer's polynomial.
The Mahler measure $M(L) = \Omega$ of Lehmer's polynomial is the smallest known for a monic integer polynomial and Lehmer's conjecture states that this is in fact \emph{the} smallest possible; see Smyth's expository article~\cite{Smyth:MMsurvey08} for a discussion of the conjecture.
Using cyclotomic matrices, Lehmer's conjecture has been verified for classes of polynomials coming from Hermitian matrices over certain rings of integers \cite{McKee:noncycISM09,GTay:Lehmer11}.

Cyclotomic matrices have been implicitly studied in spectral graph theory.
In 1970, Smith \cite{Smith:CycloG} obtained a classification of all graphs (whose adjacency matrices are symmetric $\{0,1\}$-matrices with only zeros on the diagonal) having largest eigenvalue at most $2$.
Effectively, Smith classified cyclotomic $\{0,1\}$-matrices by showing that each one is a principal submatrix of an adjacency matrix of one of the graphs $\tilde A_n$, $\tilde D_n$, $\tilde E_6$, $\tilde E_7$, and $\tilde E_8$.
These ADE graphs are the ubiquitous simply-laced affine Dynkin diagrams, see Bourbaki's book \cite{Bour:LieGroups} for their description. 
These graphs turned out to have importance in the study of graphs with bounded spectra, see the survey by Cvetkovi\'{c} and Rowlinson~\cite{Cvet:GraphSurvey90}.
Hence, the classification of cyclotomic matrices over larger sets containing $\{0,1\}$ can be seen as a generalisation of some of this work.
Cameron, Goethals, Seidel, and Shult ~\cite{Cam:LineSystems76} classified all graphs having smallest eigenvalue $-2$; they showed that one can obtain any such graph by searching inside the root systems $A_n$, $D_n$, $E_6$, $E_7$, and $E_8$.
This idea of searching inside root systems was then used later by McKee and Smyth as a part of the classification of cyclotomic integer symmetric matrices.

Let $R_i$ be an imaginary quadratic integer ring $\mathcal O_{\Q(\sqrt{d})}$ where $d < 0$. 
For Hermitian $R_i$-matrices $A$, the integrality of the characteristic polynomial is automatic.
The nontrivial Galois automorphism $\sigma$ of $\Q(\sqrt{d})$ over $\Q$ is simply complex conjugation.
Applying $\sigma$ to the coefficients of $\chi_A$ gives $\sigma(\chi_A(x)) = \det(xI - \sigma(A)) = \det(xI - A^\top) = \chi_A(x)$.
Hence, the coefficients of $\chi_A$ are rational, and since they are also algebraic integers, they must be in $\Z$.
Therefore all Hermitian $R_i$-matrices whose eigenvalues are contained inside the interval $[-2,2]$ are cyclotomic.
However, over real quadratic integer rings, things are not so simple.
For example, the matrix
\[
	\begin{pmatrix}
		\sqrt{2} & 1 \\
		1 & 0
	\end{pmatrix}
\]
has all its eigenvalues lying in the interval $[-2,2]$ but its characteristic polynomial does not have integral coefficients, hence it is not cyclotomic.
This complication of having to worry about whether or not the characteristic polynomial is integral is the reason we treat real quadratic integer rings separately to the imaginary quadratic integer rings.
There is, though, a redeeming feature of working over subrings of the real numbers; here we have a notion of nonnegativity and we can therefore make use of Perron-Frobenius theory.

Let $\mathcal R$ be the compositum of all real quadratic integer rings $\mathcal O_{\Q(\sqrt{d})}$ where $d > 1$ is squarefree.
Given a symmetric $\mathcal R$-matrix $A$, let $L_A$ denote the smallest normal extension of $\Q$ that contains all the entries of $A$.
We define $\mathfrak S_n^\prime$ to be the set of $n \times n$ symmetric $\mathcal R$-matrices $A$ such that the spectrum of $\sigma(A)$ is contained in $[-2,2]$ for all $\sigma \in \operatorname{Gal}(L_A / \Q)$.
We also define a finer set $\mathfrak S_n$ as the set of matrices from $\mathfrak S^\prime_n$ having integral characteristic polynomials.
It is clear from the above example that $\mathfrak S_2^\prime$ \emph{strictly} contains $\mathfrak S_2$.
Notice that $\mathfrak S_n$ is precisely the set of $n \times n$ cyclotomic $\mathcal R$-matrices.
In section~\ref{sec:equalsets}, we show that for $n > 6$ the two sets $\mathfrak S_n$ and $\mathfrak S^\prime_n$ are equal.

The article is organised as follows.
In Section~\ref{sec:useful} we set up the notion of equivalence used in the classification and reduce the problem to considering only indecomposable matrices.
Our results are then stated in Section~\ref{sec:results}, and proved in Section~\ref{sec:pf}.
Note that the bulk of the classification of cyclotomic matrices over real quadratic integer rings follows from the author's previous paper \cite{Greaves:CycloEG11} and we merely allude to the proof of Theorem~\ref{thm:class} in Section~\ref{sec:results}.

\section{Equivalence and interlacing}

\label{sec:useful}

In this section, we describe the equivalence classes that are used in the classification.
Our definition of equivalence is a natural extension of that used in the classification of cyclotomic $\Z$-matrices \cite{McKee:IntSymCyc07}.
Let $\mathcal R$ be the compositum of real quadratic integer rings and let $R^\prime$ be some finite subset of $\mathcal R$.
Let $R$ be the ring generated by the elements of $R^\prime$ over $\Z$ and let $K$ be the normal closure of the field generated by the elements of $R^\prime$ over $\Q$.
We write $M_n(R)$ for the ring of $n \times n$ matrices over the ring $R$. 
Denote by $O_n(\Z)$ the orthogonal group of matrices $Q$ in $M_n(\Z)$ which satisfy $QQ^\top = Q^\top Q = I$, where $Q^\top$ is the transpose of $Q$.
Let $M$ be a matrix from $M_n(R)$.
Conjugation by a matrix in $O_n(\Z)$ preserves the eigenvalues and the resulting matrix remains in $M_n(R)$.

If $\chi_M(x) \in \Z[x]$ then, since they are rational integers, the coefficients of the characteristic polynomial of $M$ are invariant under the action of $\operatorname{Gal}(K/\Q)$.
For $A$ and $B$ in $M_n(R)$, we say that $A$ is \textbf{strongly equivalent} to $B$ if $A = \sigma(QBQ^\top)$ for some $Q \in O_n(\Z)$ and some $\sigma \in \operatorname{Gal}(K/\Q)$, where $\sigma$ is applied componentwise to $QBQ^\top$.
The matrices $A$ and $B$ are merely called \textbf{equivalent} if $A$ is strongly equivalent to $\pm B$.
Note that, given a matrix $M \in \mathfrak S^\prime_n \backslash \mathfrak S_n$ all the elements of its equivalence class are also in $\mathfrak S^\prime_n \backslash \mathfrak S_n$.
Also, any two strongly equivalent matrices from $\mathfrak S_n$ have the same set of eigenvalues.

We use graphs as a convenient representation of an equivalence class of matrices.
An \textbf{$R$-graph} $G$ is an undirected weighted graph $(G,w)$ whose weight function $w$ maps pairs of vertices to elements of $R$ and satisfies $w(u,v) = w(v,u)$ for all vertices $u,v \in V(G)$.
The adjacency matrix $A = (a_{uv})$ of $G$ has $a_{uv} = w(u,v)$.
If $G$ is a triangle/cycle/tree/path/etc.\ then we call $G$ an $R$-triangle/cycle/tree/path/etc.
For every vertex $v$, the \textbf{charge} of $v$ is just the number $w(v,v)$.
A vertex with nonzero charge is called \textbf{charged}, those with zero charge are called \textbf{uncharged}.
By simply saying ``$G$ is a graph,'' we mean that $G$ is a $T$-graph where $T$ is some unspecified subset of the real numbers.

Now, $O_n(\Z)$ is generated by permutation matrices and diagonal matrices of the form
\[
	\operatorname{diag}(1,\dots,1,-1,1,\dots,1).
\]
Let $D$ be such a diagonal matrix having $-1$ in the $j$-th position.
Conjugation by $D$ is called a \textbf{switching} at vertex $j$.
A switching at vertex $j$ has the effect of multiplying all the incident edge-weights of $j$ by $-1$.
This notion of switching has been seen before as \emph{Seidel switching} \cite{Seid:Signed94}.
The effect of conjugation by permutation matrices is just a relabelling of the vertices of the corresponding graph.
Since all possible vertex-labellings of a graph are strongly equivalent, we do not have cause to label the vertices of our graphs.
The notions of equivalence and strong equivalence carry through to graphs naturally.
Throughout this paper we will interchangeably speak of both graphs and their adjacency matrices.



Next, we state a theorem of Cauchy \cite{Cau:Interlace,Fisk:Interlace05, Hwang:Interlace04} which we refer to as \emph{interlacing}.

\begin{theorem}[Interlacing] \label{thm:interlacing} Let $A$ be an $n~\times~n$ real symmetric matrix with eigenvalues $\lambda_1 \leqslant \dots \leqslant \lambda_n$. Let $B$ be an $(n - 1)~\times~(n - 1)$ principal submatrix of $A$ with eigenvalues $\mu_1 \leqslant \dots \leqslant \mu_{n-1}$. Then the eigenvalues of $A$ and $B$ interlace. Namely,
	\[
		\lambda_1 \leqslant \mu_1 \leqslant \lambda_2 \leqslant \mu_2 \leqslant \dots \leqslant \mu_{n-1} \leqslant \lambda_n.
	\]
\end{theorem}

Define the \textbf{degree} of a vertex $v \in V(G)$ as
\[
	\sum_{u \in V(G)} w(u,v)^2.
\]

\begin{lemma}\label{lem:maxDeg4}
	Let $G$ be a graph with a vertex $v$ of degree $d > 4$.
	Then $G$ does not correspond to any matrix in $\mathfrak S^\prime_n$.
\end{lemma}
\begin{proof}
	Let $A$ be an adjacency matrix of $G$ with $v$ corresponding to the first row.
	The first entry of the first row of $A^2$ is $d$.
	Therefore, by interlacing, the largest eigenvalue of $A^2$ is at least $d$, and so the largest modulus of the eigenvalues of $A$ is at least $|\sqrt{d}| > 2$.
\end{proof}

This lemma also restricts the possible entries of matrices in $\mathfrak S^\prime_n$; any entry of such a matrix must square to at most $4$.

A matrix that is equivalent to a block diagonal matrix of more than one block is called \textbf{decomposable}, otherwise it is called \textbf{indecomposable}.
A matrix $A = (a_{ij})$ is indecomposable if and only if its underlying graph (whose vertices $u$ and $v$ are adjacent if and only if $a_{uv}$ is nonzero) is connected.
The eigenvalues of a decomposable matrix are found by pooling together the eigenvalues of its blocks.
It is therefore sufficient to restrict our classification of cyclotomic matrices to indecomposable matrices.
Hence we redefine $\mathfrak S^\prime_n$ and $\mathfrak S_n$ to consist only of indecomposable matrices.
An indecomposable cyclotomic matrix that is not a principal submatrix of any other indecomposable cyclotomic matrix is called a \textbf{maximal} indecomposable cyclotomic matrix.
The corresponding graph is called a maximal connected cyclotomic graph.

We briefly return to the problem of whether or not a matrix has an integral characteristic polynomial.
It is possible to ensure the integrality of a matrix by giving its associated graph a certain symmetry.
Let $K$ be a Galois extension of $\Q$ with $R$ its ring of integers.  
We say that a Hermitian $R$-matrix is \textbf{Galois invariant} if it is strongly equivalent to itself under Galois conjugation, i.e., for all $\sigma \in \operatorname{Gal}(K/\Q)$, $G$ is strongly equivalent to $\sigma(G)$.

\begin{proposition}\label{pro:Ginvariance}
	Let $A$ be a Galois-invariant symmetric $R$-matrix.
	Then its characteristic polynomial $\chi_A$ has integer coefficients.
\end{proposition}
\begin{proof} 
	For all $\sigma \in \operatorname{Gal}(K/\Q)$, applying $\sigma$ to the coefficients of $\chi_A$ gives
	\[ 
	\sigma (\chi_A(x)) = \det( xI - \sigma(A) ) = \det( xI - A ) = \chi_A(x).
	\] 
	Hence, the characteristic polynomial $\chi_A$ must have rational coefficients. 
	And since the entries of $A$ are algebraic integers, so too are the coefficients of $\chi_A$.
\end{proof}

We observed earlier that all Hermitian matrices over imaginary quadratic integer rings are Galois invariant. 
It can be readily seen below, in the classification of $\mathcal R$-matrices, that the converse of the proposition does not hold.
For example, the maximal cyclotomic $\mathcal R$-graph $S_4^{(2,\varphi)}$ of Figure~\ref{fig:maxcycs8} is not Galois invariant, in fact, it is the only such example; all other maximal cyclotomic $\mathcal R$-graphs are Galois invariant.

%

\section{Results}
\label{sec:results}

Before stating our results, we outline our graph drawing conventions.
We draw edges with edge-weight $w$ as \tikz { \path[pedge] (0,0) -- node[weight] {$w$} (1.2,0); } and edges of weight $-w$ as \tikz { \path[nedge] (0,0) -- node[weight] {$w$} (1.2,0); }.
If $w=1$, we simply draw a solid line \tikz[auto] {\path[pedge] (0,0) -- (1.2,0); } and a dashed line \tikz {\path[nedge] (0,0) -- (1.2,0); } respectively.
A vertex with charge $c$ for some $c > 0$ is drawn as \tikz {\node[pc] {$c$};} and a vertex with charge $-c$ is drawn as \tikz {\node[nc] {$c$};}.
And if a vertex is uncharged, we simply draw \tikz {\node[zc] {};}.
By a \textbf{subgraph} $H$ of $G$ we mean an induced subgraph: a subgraph obtained by deleting vertices and their incident edges.
We say that $G$ \textbf{contains} $H$ and that $G$ is a \textbf{supergraph} of $H$.
A graph is called \textbf{charged} if it contains at least one charged vertex, otherwise it is called \textbf{uncharged}.

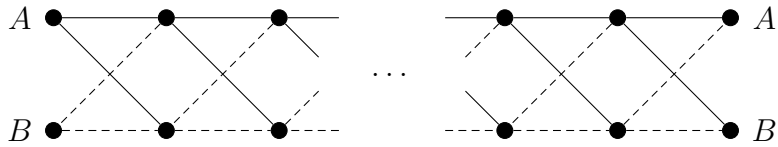
\begin{figure}[htbp]
	\centering
		\begin{tikzpicture}[auto, scale=1.5]
			\begin{scope}
				\foreach \type/\pos/\name in {{vertex/(0,0)/a2}, {vertex/(0,1)/a1}, {vertex/(1,1)/b1}, {vertex/(1,0)/b2}, {vertex/(2,0)/e2}, {vertex/(2,1)/e1}, {empty/(2.6,1)/b11}, {empty/(2.6,0)/b21}, {empty/(2.4,0.6)/b12}, {empty/(2.4,0.4)/b22}, {empty/(3.4,1)/c11}, {empty/(3.4,0)/c21}, {empty/(3.6,0.6)/c12}, {empty/(3.6,0.4)/c22}, {vertex/(4,1)/c1}, {vertex/(4,0)/c2}, {vertex/(5,1)/d1}, {vertex/(5,0)/d2}, {vertex/(6,1)/f1}, {vertex/(6,0)/f2}}
					\node[\type] (\name) at \pos {};
				\foreach \pos/\name in {{(3,0.5)/\dots}, {(-.3,1)/A}, {(-0.3,0)/B}, {(6.3,1)/A}, {(6.3,0)/B}}
					\node at \pos {$\name$};
				\foreach \edgetype/\source/ \dest/\weight in {nedge/b1/a2/{}, pedge/a1/b1/{}, pedge/a1/b2/{}, nedge/a2/b2/{}, nedge/e1/b2/{}, pedge/b1/e1/{}, pedge/b1/e2/{}, nedge/b2/e2/{}, nedge/b21/e2/{}, pedge/e1/b11/{}, pedge/e1/b12/{}, nedge/e2/b22/{}, pedge/c11/c1/{}, nedge/c12/c1/{}, pedge/c22/c2/{}, nedge/c21/c2/{}, nedge/d1/c2/{}, pedge/c1/d1/{}, pedge/c1/d2/{}, nedge/c2/d2/{}, pedge/d1/f1/{}, nedge/d2/f2/{}}
					\path[\edgetype] (\source) -- (\dest);
			\end{scope}
			\begin{scope}
				\foreach \edgetype/\source/ \dest/\weight in {nedge/d2/f1/{}, pedge/d1/f2/{}}
					\path[\edgetype] (\source) -- (\dest);
			\end{scope}
		\end{tikzpicture}
	\caption{The family $T_{2k}$ of $2k$-vertex maximal connected cyclotomic $\Z$-graphs, for $k \geqslant 3$. (The two copies of vertices $A$ and $B$ should be identified to give a toral tessellation.) }
	\label{fig:maxcycs1}
\end{figure}

\begin{figure}[htbp]
	\centering
		\begin{tikzpicture}[scale=1.5]
			\begin{scope}
				\foreach \type/\pos/\name in {{vertex/(1,1)/b1}, {vertex/(1,0)/b2}, {vertex/(2,0)/e2}, {vertex/(2,1)/e1}, {empty/(2.6,1)/b11}, {empty/(2.6,0)/b21}, {empty/(2.4,0.6)/b12}, {empty/(2.4,0.4)/b22}, {empty/(3.4,1)/c11}, {empty/(3.4,0)/c21}, {empty/(3.6,0.6)/c12}, {empty/(3.6,0.4)/c22}, {vertex/(4,1)/c1}, {vertex/(4,0)/c2}, {vertex/(5,1)/d1}, {vertex/(5,0)/d2}}
					\node[\type] (\name) at \pos {};
				\foreach \type/\pos/\name in {{vertex/(0,0.5)/bgn}, {vertex/(6,0.5)/end}}
					\node[\type] (\name) at \pos {};
				\foreach \pos/\name in {{(3,0.5)/\dots}}
					\node at \pos {\name};
				\foreach \edgetype/\source/ \dest in {nedge/e1/b2, pedge/b1/e1, pedge/b1/e2, nedge/b2/e2, nedge/b21/e2, pedge/e1/b11, pedge/e1/b12, nedge/e2/b22, pedge/c11/c1, nedge/c12/c1, pedge/c22/c2, nedge/c21/c2, nedge/d1/c2, pedge/c1/d1, pedge/c1/d2, nedge/c2/d2}
					\path[\edgetype] (\source) -- (\dest);
				\foreach \edgetype/\source/\dest in {pedge/b1/bgn, pedge/b2/bgn, nedge/d2/end, pedge/d1/end}
					\path[\edgetype] (\source) -- node[weight] {$\sqrt{2}$} (\dest);
			\end{scope}
		\end{tikzpicture}
	\caption{The family of $2k$-vertex maximal connected cyclotomic $\Z[\sqrt{2}]$-graphs $C_{2k}$ for $k \geqslant 2$.}
	\label{fig:maxcycs2}
\end{figure}
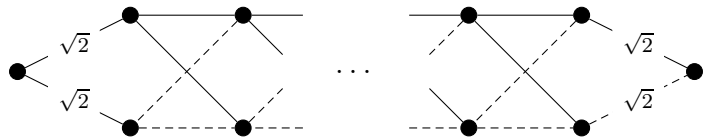

\begin{figure}[htbp]
	\centering
		\begin{tikzpicture}[scale=1.5, auto]
				\foreach \type/\pos/\name in {{vertex/(1,1)/b1}, {vertex/(1,0)/b2}, {vertex/(2,0)/e2}, {vertex/(2,1)/e1}, {empty/(2.6,1)/b11}, {empty/(2.6,0)/b21}, {empty/(2.4,0.6)/b12}, {empty/(2.4,0.4)/b22}, {empty/(3.4,1)/c11}, {empty/(3.4,0)/c21}, {empty/(3.6,0.6)/c12}, {empty/(3.6,0.4)/c22}, {vertex/(4,1)/c1}, {vertex/(4,0)/c2}, {vertex/(5,1)/d1}, {vertex/(5,0)/d2}}
					\node[\type] (\name) at \pos {};
				\foreach \type/\pos/\name in {{pc/(0,0)/a2}, {pc/(0,1)/a1}, {pc/(6,1)/f1}, {pc/(6,0)/f2}}
					\node[\type] (\name) at \pos {$1$};
				\foreach \pos/\name in {{(3,0.5)/\dots}}
					\node at \pos {$\name$};
				\foreach \edgetype/\source/ \dest in {pedge/a1/a2, nedge/b1/a2, pedge/a1/b1, pedge/a1/b2, nedge/a2/b2, nedge/e1/b2, pedge/b1/e1, pedge/b1/e2, nedge/b2/e2, nedge/b21/e2, pedge/e1/b11, pedge/e1/b12, nedge/e2/b22, pedge/c11/c1, nedge/c12/c1, pedge/c22/c2, nedge/c21/c2, nedge/d1/c2, pedge/c1/d1, pedge/c1/d2, nedge/c2/d2, nedge/f1/d2, pedge/d1/f1, pedge/d1/f2, nedge/d2/f2, nedge/f1/f2}
					\path[\edgetype] (\source) -- (\dest);
				\node at (3,-0.2) {};
			\end{tikzpicture}
			\begin{tikzpicture}[scale=1.5, auto]
				\foreach \type/\pos/\name in {{vertex/(1,1)/b1}, {vertex/(1,0)/b2}, {vertex/(2,0)/e2}, {vertex/(2,1)/e1}, {empty/(2.6,1)/b11}, {empty/(2.6,0)/b21}, {empty/(2.4,0.6)/b12}, {empty/(2.4,0.4)/b22}, {empty/(3.4,1)/c11}, {empty/(3.4,0)/c21}, {empty/(3.6,0.6)/c12}, {empty/(3.6,0.4)/c22}, {vertex/(4,1)/c1}, {vertex/(4,0)/c2}, {vertex/(5,1)/d1}, {vertex/(5,0)/d2}}
					\node[\type] (\name) at \pos {};
				\foreach \type/\pos/\name in {{pc/(0,0)/a2}, {pc/(0,1)/a1}, {nc/(6,1)/f1}, {nc/(6,0)/f2}}
					\node[\type] (\name) at \pos {$1$};
				\foreach \pos/\name in {{(3,0.5)/\dots}}
					\node at \pos {$\name$};
				\foreach \edgetype/\source/ \dest in {pedge/a1/a2, nedge/b1/a2, pedge/a1/b1, pedge/a1/b2, nedge/a2/b2, nedge/e1/b2, pedge/b1/e1, pedge/b1/e2, nedge/b2/e2, nedge/b21/e2, pedge/e1/b11, pedge/e1/b12, nedge/e2/b22, pedge/c11/c1, nedge/c12/c1, pedge/c22/c2, nedge/c21/c2, nedge/d1/c2, pedge/c1/d1, pedge/c1/d2, nedge/c2/d2, nedge/f1/d2, pedge/d1/f1, pedge/d1/f2, nedge/d2/f2, pedge/f1/f2}
					\path[\edgetype] (\source) -- (\dest);
		\end{tikzpicture}
	\caption{The families of $2k$-vertex maximal connected cyclotomic $\Z$-graphs $C_{2k}^{++}$ and $C_{2k}^{+-}$ for $k \geqslant 2$.}
	\label{fig:maxcycs3}
\end{figure}
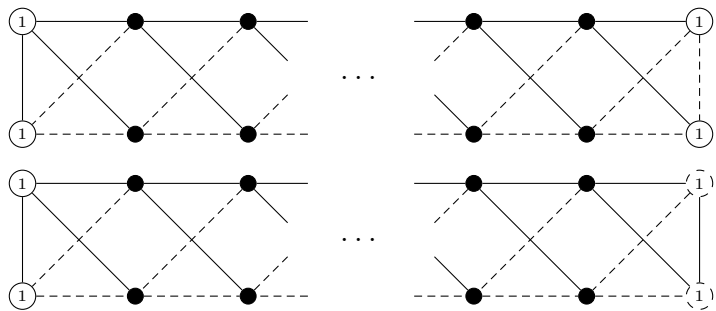

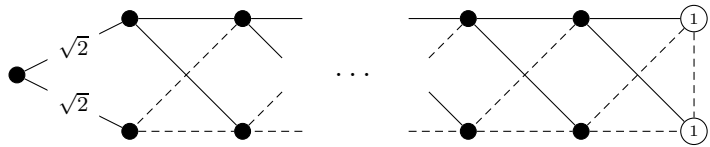
\begin{figure}[htbp]
	\centering
		\begin{tikzpicture}[scale=1.5]
			\begin{scope}
				\foreach \type/\pos/\name in {{vertex/(0,0.5)/bgn}, {vertex/(1,1)/b1}, {vertex/(1,0)/b2}, {vertex/(2,0)/e2}, {vertex/(2,1)/e1}, {empty/(2.6,1)/b11}, {empty/(2.6,0)/b21}, {empty/(2.4,0.6)/b12}, {empty/(2.4,0.4)/b22}, {empty/(3.4,1)/c11}, {empty/(3.4,0)/c21}, {empty/(3.6,0.6)/c12}, {empty/(3.6,0.4)/c22}, {vertex/(4,1)/c1}, {vertex/(4,0)/c2}, {vertex/(5,1)/d1}, {vertex/(5,0)/d2}}
					\node[\type] (\name) at \pos {};
				\foreach \type/\pos/\name in {{pc/(6,1)/f1}, {pc/(6,0)/f2}}
					\node[\type] (\name) at \pos {$1$};
				\foreach \pos/\name in {{(3,0.5)/\dots}}
					\node at \pos {$\name$};
				\foreach \edgetype/\source/ \dest in {nedge/e1/b2, pedge/b1/e1, pedge/b1/e2, nedge/b2/e2, nedge/b21/e2, pedge/e1/b11, pedge/e1/b12, nedge/e2/b22, pedge/c11/c1, nedge/c12/c1, pedge/c22/c2, nedge/c21/c2, nedge/d1/c2, pedge/c1/d1, pedge/c1/d2, nedge/c2/d2, nedge/f1/d2, pedge/d1/f1, pedge/d1/f2, nedge/d2/f2, nedge/f1/f2}
					\path[\edgetype] (\source) -- (\dest);
				\foreach \edgetype/\source/\dest in {pedge/b1/bgn, pedge/b2/bgn}
					\path[\edgetype] (\source) -- node[weight] {$\sqrt{2}$} (\dest);
			\end{scope}
		\end{tikzpicture}
	\caption{The family of $(2k+1)$-vertex maximal connected cyclotomic $\Z[\sqrt{2}]$-graphs $C_{2k+1}$ for $k \geqslant 1$.}
	\label{fig:maxcycs4}
\end{figure}

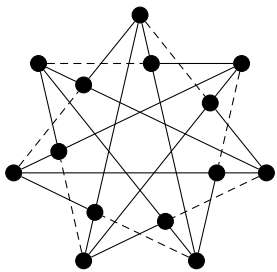
\begin{figure}[htbp]
	\centering
		\begin{tikzpicture}
			\begin{scope}[auto, scale=1.5]
				\foreach \type/\pos/\name in {{vertex/(0,0)/a}, {vertex/(1,0)/b}, {vertex/(1.62,0.78)/c}, {vertex/(1.4,1.75)/d}, {vertex/(0.5,2.18)/e}, {vertex/(-0.4,1.75)/f}, {vertex/(-0.62,0.78)/g}, {vertex/(0.1,0.43)/t}, {vertex/(0.725,0.35)/u}, {vertex/(1.18,0.78)/v}, {vertex/(1.122,1.4)/w}, {vertex/(0.6,1.75)/x}, {vertex/(0,1.56)/y}, {vertex/(-0.22,0.97)/z}}
					\node[\type] (\name) at \pos {};
				\foreach \edgetype/\source/ \dest in {pedge/a/e, nedge/a/z, pedge/a/d, pedge/a/u, nedge/b/t, pedge/b/f, pedge/b/e, pedge/b/v, nedge/c/u, pedge/c/g, pedge/c/f, pedge/c/w, nedge/d/v, pedge/d/g, pedge/d/x, nedge/e/w, pedge/e/y, nedge/f/x, pedge/f/z, nedge/g/y, pedge/g/t}
					\path[\edgetype] (\source) -- (\dest);
			\end{scope}
		\end{tikzpicture}
 	\caption{The sporadic maximal connected cyclotomic $\Z$-graph $S_{14}$ of order $14$.}
	\label{fig:maxcycs5}
\end{figure}

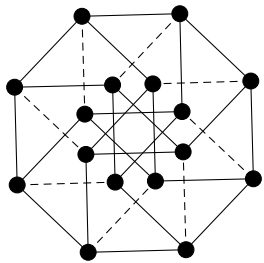
\begin{figure}[htbp]
	\centering
	\begin{tikzpicture}
		\newdimen\rad
		\rad=1.7cm
		\newdimen\radi
		\radi=0.7cm
		\foreach \x in {69,114,159,204,249,294,339,384}
		{
			\draw (\x:\radi) node[vertex] {};
			\draw[pedge] (\x:\radi) -- (\x+135:\radi);
	    }
		\foreach \x in {69,159,249,339}
		{
	    	\draw (\x:\rad) node[vertex] {};
			\draw[pedge] (\x:\rad) -- (\x+45:\rad);
			\draw[nedge] (\x:\rad) -- (\x+45:\radi);
			\draw[pedge] (\x:\rad) -- (\x-45:\radi);
	    }
		\foreach \x in {114,204,294,384}
		{
	    	\draw (\x:\rad) node[vertex] {};
			\draw[pedge] (\x:\rad) -- (\x+45:\rad);
			\draw[nedge] (\x:\rad) -- (\x+45:\radi);
			\draw[pedge] (\x:\rad) -- (\x-45:\radi);
	    }
	\end{tikzpicture}
	\caption{The sporadic maximal connected cyclotomic $\Z$-hypercube $S_{16}$.}
	\label{fig:maxcycs6}
\end{figure}

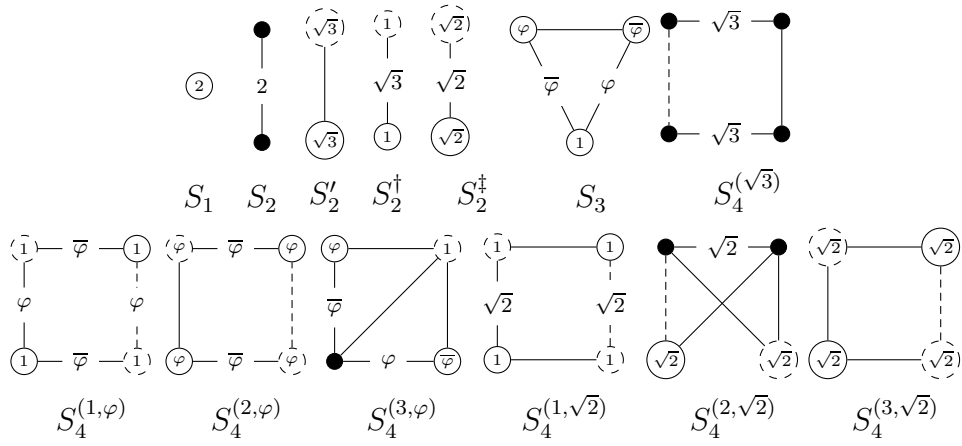
\begin{figure}[htbp]
	\centering
	\begin{tikzpicture}[scale=1.5, auto]
	\begin{scope}	
		\foreach \pos/\name/\type/\charge in {{(0,0.4)/a/pc/{2}}}
			\node[\type] (\name) at \pos {$\charge$}; 
		\node at (0,-0.6) {$S_1$};
	\end{scope}
	\end{tikzpicture}
	\begin{tikzpicture}[scale=1.5]
	\begin{scope}	
		\foreach \pos/\name/\type/\charge in {{(0,0)/a/zc/{}}, {(0,1)/b/zc/{}}}
			\node[\type] (\name) at \pos {$\charge$}; 
		\foreach \edgetype/\source/ \dest in {pedge/a/b}
		\path[\edgetype] (\source) -- node[weight] {$2$} (\dest);
		\node at (0,-0.5) {$S_2$};
	\end{scope}
	\end{tikzpicture}
	\begin{tikzpicture}[scale=1.5]
		\foreach \pos/\name/\sign/\charge in {{(0,0)/a/pc/\sqrt{3}}, {(0,1)/b/nc/\sqrt{3}}}
			\node[\sign] (\name) at \pos {$\charge$}; 
		\node at (0,-0.5) {$S_2^\prime$};
		\foreach \edgetype/\source/ \dest /\weight in {pedge/b/a/{}}
		\path[\edgetype] (\source) -- (\dest);
	\end{tikzpicture}
	\begin{tikzpicture}[scale=1.5]
		\foreach \pos/\name/\sign/\charge in {{(0,0)/a/pc/1}, {(0,1)/b/nc/1}}
			\node[\sign] (\name) at \pos {$\charge$}; 
		\node at (0,-0.5) {$S_2^{\dag}$};
		\foreach \edgetype/\source/ \dest /\weight in {pedge/b/a/{\sqrt{3}}}
		\path[\edgetype] (\source) -- node[weight] {$\weight$} (\dest);
	\end{tikzpicture}
	\begin{tikzpicture}[scale=1.5]	
		\foreach \pos/\name/\sign/\charge in {{(0,0)/a/pc/\sqrt{2}}, {(0,1)/b/nc/\sqrt{2}}}
			\node[\sign] (\name) at \pos {$\charge$}; 
		\node at (0.2,-0.5) {$S_2^\ddag$};
		\foreach \edgetype/\source/ \dest /\weight in {pedge/b/a/\sqrt{2}}
		\path[\edgetype] (\source) -- node[weight] {$\weight$} (\dest);
	\end{tikzpicture}
	\begin{tikzpicture}[scale=1.5]
			\foreach \pos/\name/\sign/\charge in {{(0,1)/a/pc/\varphi}, {(1,1)/b/pc/\bar\varphi}, {(0.5,0)/c/pc/1}}
				\node[\sign] (\name) at \pos {$\charge$}; 
			\node at (0.6,-0.5) {$S_3$};
			\foreach \edgetype/\source/ \dest /\weight in {pedge/a/c/\bar\varphi, pedge/c/b/\varphi}
			\path[\edgetype] (\source) -- node[weight] {$\weight$} (\dest);
			\foreach \edgetype/\source/ \dest /\weight in {pedge/b/a/{}}
			\path[\edgetype] (\source) -- (\dest);
		\end{tikzpicture}
		\begin{tikzpicture}[scale=1.5]
			\foreach \pos/\name in {{(0,0)/a}, {(1,0)/b}, {(0,1)/c}, {(1,1)/d}}
				\node[vertex] (\name) at \pos {};
			\node at (0.7,-0.5) {$S_4^{(\sqrt{3})}$};
			\foreach \edgetype/\source/ \dest /\weight in {pedge/a/b/\sqrt{3}, pedge/d/c/\sqrt{3}}
				\path[\edgetype] (\source) -- node[weight] {$\weight$} (\dest);
				\foreach \edgetype/\source/ \dest /\weight in {pedge/b/d/{}, nedge/c/a/{}}
					\path[\edgetype] (\source) -- (\dest);
		\end{tikzpicture}
		
		\begin{tikzpicture}[scale=1.5]
			\foreach \pos/\name/\sign/\charge in {{(0,0)/a/pc/1}, {(0,1)/b/nc/1}, {(1,0)/c/nc/1}, {(1,1)/d/pc/1}}
				\node[\sign] (\name) at \pos {$\charge$}; 
			\node at (0.6,-0.5) {$S_4^{(1, \varphi)}$};
			\foreach \edgetype/\source/ \dest /\weight in {pedge/b/a/\varphi, pedge/a/c/\bar\varphi, pedge/d/b/\bar\varphi, nedge/c/d/\varphi}
			\path[\edgetype] (\source) -- node[weight] {$\weight$} (\dest);
		\end{tikzpicture}
		\begin{tikzpicture}[scale=1.5]
			\foreach \pos/\name/\sign/\charge in {{(0,0)/a/pc/\varphi}, {(0,1)/b/nc/\varphi}, {(1,0)/c/nc/\varphi}, {(1,1)/d/pc/\varphi}}
				\node[\sign] (\name) at \pos {$\charge$}; 
			\node at (0.6,-0.5) {$S_4^{(2, \varphi)}$};
			\foreach \edgetype/\source/ \dest /\weight in {pedge/a/c/\bar\varphi, pedge/d/b/\bar\varphi}
			\path[\edgetype] (\source) -- node[weight] {$\weight$} (\dest);
			\foreach \edgetype/\source/ \dest /\weight in {pedge/b/a/{}, nedge/c/d/{}}
			\path[\edgetype] (\source) -- (\dest);
		\end{tikzpicture}
		\begin{tikzpicture}[scale=1.5]
			\foreach \pos/\name/\sign/\charge in {{(0,0)/a/zc/{}}, {(0,1)/b/pc/\varphi}, {(1,0)/c/pc/\bar\varphi}, {(1,1)/d/nc/1}}
				\node[\sign] (\name) at \pos {$\charge$}; 
			\node at (0.6,-0.5) {$S_4^{(3, \varphi)}$};
			\foreach \edgetype/\source/ \dest /\weight in {pedge/b/a/\bar\varphi, pedge/a/c/\varphi}
			\path[\edgetype] (\source) -- node[weight] {$\weight$} (\dest);
			\foreach \edgetype/\source/ \dest /\weight in {pedge/d/b/{}, pedge/c/d/{}, pedge/a/d/{}}
			\path[\edgetype] (\source) -- (\dest);
		\end{tikzpicture}
		\begin{tikzpicture}[scale=1.5]
		\foreach \pos/\name/\sign/\charge in {{(0,0)/a/pc/1}, {(0,1)/b/nc/1}, {(1,0)/c/nc/1}, {(1,1)/d/pc/1}}
			\node[\sign] (\name) at \pos {$\charge$}; 
		\node at (0.6,-0.5) {$S_4^{(1, \sqrt{2})}$};
		\foreach \edgetype/\source/ \dest /\weight in {pedge/b/a/\sqrt{2}, nedge/c/d/\sqrt{2}}
		\path[\edgetype] (\source) -- node[weight] {$\weight$} (\dest);
		\foreach \edgetype/\source/ \dest /\weight in {pedge/d/b/{}, pedge/a/c/{}}
			\path[\edgetype] (\source) -- (\dest);
	\end{tikzpicture}
	\begin{tikzpicture}[scale=1.5]
		\foreach \pos/\name/\sign/\charge in {{(0,0)/a/pc/\sqrt{2}}, {(0,1)/b/zc/{}}, {(1,0)/c/nc/\sqrt{2}}, {(1,1)/d/zc/{}}}
			\node[\sign] (\name) at \pos {$\charge$}; 
		\node at (0.6,-0.5) {$S_4^{(2, \sqrt{2})}$};
		\foreach \edgetype/\source/ \dest /\weight in {pedge/d/b/\sqrt{2}}
		\path[\edgetype] (\source) -- node[weight] {$\weight$} (\dest);
		\foreach \edgetype/\source/ \dest /\weight in {nedge/b/a/{}, pedge/c/d/{}, pedge/c/b/{}, pedge/a/d/{}}
		\path[\edgetype] (\source) -- (\dest);
	\end{tikzpicture}
	\begin{tikzpicture}[scale=1.5]	
		\foreach \pos/\name/\sign/\charge in {{(0,0)/a/pc/\sqrt{2}}, {(0,1)/b/nc/\sqrt{2}}, {(1,0)/c/nc/\sqrt{2}}, {(1,1)/d/pc/\sqrt{2}}}
			\node[\sign] (\name) at \pos {$\charge$}; 
		\node at (0.6,-0.5) {$S_4^{(3, \sqrt{2})}$};
		\foreach \edgetype/\source/ \dest /\weight in {pedge/b/a/{}, nedge/c/d/{}, pedge/d/b/{}, pedge/a/c/{}}
		\path[\edgetype] (\source) -- (\dest);
	\end{tikzpicture}
	\caption{The sporadic maximal connected cyclotomic $\mathcal R$-graphs of orders $1$, $2$, $3$ and $4$.}
	\label{fig:maxcycs8}
\end{figure}

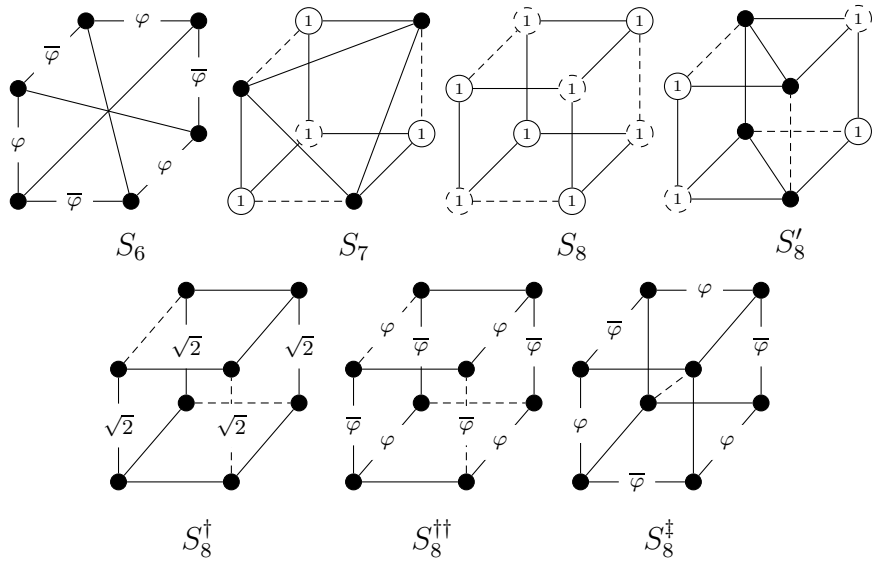
\begin{figure}[htbp]
	\centering
	\begin{tikzpicture}[scale=1.5]
			\def\XthreeDadj{0.6}
			\def\YthreeDadj{0.6}
		\foreach \pos/\name in {{(0,0)/a}, {(0,1)/b}, {(1 + \XthreeDadj,\YthreeDadj)/c}, {(1,0)/d}, {(\XthreeDadj,1+\YthreeDadj)/e}, {(1+\XthreeDadj,1+\YthreeDadj)/f}}
			\node[vertex] (\name) at \pos {};
		\node at (1,-0.4) {$S_6$};
		\foreach \edgetype/\source/ \dest /\weight in {pedge/a/b/\varphi, pedge/b/e/\bar\varphi, pedge/d/a/\bar\varphi, pedge/c/d/\varphi, pedge/f/c/\bar\varphi, pedge/e/f/\varphi}
		\path[\edgetype] (\source) -- node[weight] {$\weight$} (\dest);
		\foreach \edgetype/\source/ \dest in {pedge/b/c, pedge/d/e, pedge/f/a}
		\path[\edgetype] (\source) -- (\dest);
	\end{tikzpicture}
	\begin{tikzpicture}[scale=1.5]
		\def\XthreeDadj{0.6}
		\def\YthreeDadj{0.6}
			\foreach \pos/\name/\sign/\charge in {{(0,0)/a/pc/1}, {(0,1)/b/zc/{}}, {(1,0)/c/zc/{}}, {(0 + \XthreeDadj,0 + \YthreeDadj)/d/nc/1}, {(0 + \XthreeDadj,1 + \YthreeDadj)/e/pc/1}, {(1 + \XthreeDadj,0 + \YthreeDadj)/f/pc/1}, {(1 + \XthreeDadj,1 + \YthreeDadj)/g/zc/{}}}
				\node[\sign] (\name) at \pos {$\charge$}; 
			\foreach \edgetype/\source/ \dest in {pedge/b/a, nedge/a/c, pedge/a/d, pedge/c/f, nedge/g/f, pedge/b/g, pedge/d/e, pedge/d/f, nedge/b/e, pedge/c/g, pedge/e/g, pedge/b/c}
			\path[\edgetype] (\source) -- (\dest);
		\node at (1,-0.4) {$S_{7}$};
		\end{tikzpicture}
		\begin{tikzpicture}[scale=1.5]
			\def\XthreeDadj{0.6}
			\def\YthreeDadj{0.6}
		\begin{scope}
			\foreach \pos/\name/\sign/\charge in {{(0,0)/a/nc/1}, {(0,1)/b/pc/1}, {(1,0)/c/pc/1}, {(1,1)/d/nc/1}, {(0 + \XthreeDadj,0 + \YthreeDadj)/e/pc/1}, {(0 + \XthreeDadj,1 + \YthreeDadj)/f/nc/1}, {(1 + \XthreeDadj,0 + \YthreeDadj)/g/nc/1}, {(1 + \XthreeDadj,1 + \YthreeDadj)/h/pc/1}}
				\node[\sign] (\name) at \pos {$\charge$}; 
			\foreach \edgetype/\source/ \dest in {pedge/b/a, nedge/a/c, pedge/a/e, pedge/c/g, pedge/c/d, nedge/b/f, pedge/b/d, pedge/e/f, pedge/e/g, nedge/h/g, pedge/f/h, pedge/d/h}
			\path[\edgetype] (\source) -- (\dest);
		\node at (1,-0.4) {$S_{8}$};
		\end{scope}
		\end{tikzpicture}
		\begin{tikzpicture}[scale=1.5]
			\def\XthreeDadj{0.6}
			\def\YthreeDadj{0.6}
		\begin{scope}	
			\foreach \pos/\name/\sign/\charge in {{(0,0)/a/nc/1}, {(0,1)/b/pc/1}, {(1,0)/c/zc/{}}, {(1,1)/d/zc/{}}, {(0 + \XthreeDadj,0 + \YthreeDadj)/e/zc/{}}, {(0 + \XthreeDadj,1 + \YthreeDadj)/f/zc/{}}, {(1 + \XthreeDadj,0 + \YthreeDadj)/g/pc/1}, {(1 + \XthreeDadj,1 + \YthreeDadj)/h/nc/1}}
				\node[\sign] (\name) at \pos {$\charge$}; 
			\foreach \edgetype/\source/ \dest in {pedge/b/a, pedge/a/c, pedge/a/e, pedge/c/g, nedge/c/d, nedge/b/f, pedge/b/d, pedge/e/f, nedge/e/g, pedge/h/g, pedge/f/h, pedge/d/h, pedge/d/f, pedge/c/e}
			\path[\edgetype] (\source) -- (\dest);
		\node at (1,-0.4) {$S^\prime_{8}$};
		\end{scope}
	\end{tikzpicture}
	
	\begin{tikzpicture}[scale=1.5]
		\def\XthreeDadj{0.6}
		\def\YthreeDadj{0.7}
		\foreach \pos/\name in {{(0,0)/a}, {(0,1)/b}, {(1,0)/c}, {(1,1)/d}, {(0 + \XthreeDadj,0 + \YthreeDadj)/e}, {(0 + \XthreeDadj,1 + \YthreeDadj)/f}, {(1 + \XthreeDadj,0 + \YthreeDadj)/g}, {(1 + \XthreeDadj,1 + \YthreeDadj)/h}}
			\node[vertex] (\name) at \pos {}; 
		\node at (0.7,-0.5) {$S_8^\dag$};
		\foreach \edgetype/\source/ \dest /\weight in {pedge/b/a/\sqrt{2}, pedge/f/e/\sqrt{2}, nedge/c/d/\sqrt{2}, pedge/g/h/\sqrt{2}}
		\path[\edgetype] (\source) -- node[weight] {$\weight$} (\dest);
		\foreach \edgetype/\source/ \dest /\weight in {pedge/d/b/{}, nedge/f/b/{}, nedge/e/g/{}, pedge/c/g/{}, pedge/f/h/{}, pedge/a/e/{}, pedge/d/h/{}, pedge/a/c/{}}
		\path[\edgetype] (\source) -- (\dest);
	\end{tikzpicture}
	\begin{tikzpicture}[scale=1.5]
		\def\XthreeDadj{0.6}
		\def\YthreeDadj{0.7}
		\foreach \pos/\name in {{(0,0)/a}, {(0,1)/b}, {(1,0)/c}, {(1,1)/d}, {(0 + \XthreeDadj,0 + \YthreeDadj)/e}, {(0 + \XthreeDadj,1 + \YthreeDadj)/f}, {(1 + \XthreeDadj,0 + \YthreeDadj)/g}, {(1 + \XthreeDadj,1 + \YthreeDadj)/h}}
			\node[vertex] (\name) at \pos {};
		\node at (0.7,-0.5) {$S_8^{\dag\dag}$};
		\foreach \edgetype/\source/ \dest /\weight in {pedge/b/a/\bar\varphi, pedge/f/e/\bar\varphi, nedge/c/d/\bar\varphi, nedge/f/b/\varphi, pedge/c/g/\varphi, pedge/g/h/\bar\varphi, pedge/e/a/\varphi, pedge/d/h/\varphi}
		\path[\edgetype] (\source) -- node[weight] {$\weight$} (\dest);
		\foreach \edgetype/\source/ \dest /\weight in {pedge/d/b/{}, nedge/e/g/{}, pedge/f/h/{}, pedge/a/c/{}}
		\path[\edgetype] (\source) -- (\dest);
	\end{tikzpicture}
	\begin{tikzpicture}[scale=1.5]
		\def\XthreeDadj{0.6}
		\def\YthreeDadj{0.7}
		\foreach \pos/\name in {{(0,0)/a}, {(0,1)/b}, {(1,0)/c}, {(1,1)/d}, {(0 + \XthreeDadj,0 + \YthreeDadj)/e}, {(0 + \XthreeDadj,1 + \YthreeDadj)/f}, {(1 + \XthreeDadj,0 + \YthreeDadj)/g}, {(1 + \XthreeDadj,1 + \YthreeDadj)/h}}
			\node[vertex] (\name) at \pos {};
		\node at (0.7,-0.5) {$S_8^\ddag$};
		\foreach \edgetype/\source/ \dest /\weight in {pedge/b/a/\varphi, pedge/f/b/\bar\varphi, pedge/c/g/\varphi, pedge/g/h/\bar\varphi, pedge/h/f/\varphi, pedge/a/c/\bar\varphi}
		\path[\edgetype] (\source) -- node[weight] {$\weight$} (\dest);
		\foreach \edgetype/\source/ \dest /\weight in {pedge/e/f/{}, pedge/d/c/{}, pedge/d/b/{}, pedge/e/g/{}, nedge/e/d/{}, pedge/a/e/{}, pedge/d/h/{}}
		\path[\edgetype] (\source) -- (\dest);
	\end{tikzpicture}
	\caption{The sporadic maximal connected cyclotomic $\mathcal R$-graphs of orders $6$, $7$, and $8$.}
	\label{fig:maxcycs10}
\end{figure}

\begin{theorem}\cite{McKee:IntSymCyc07}\label{thm:ISM}
	Let $A$ be a maximal indecomposable cyclotomic matrix over the ring $\Z$.
	Then $A$ is equivalent to an adjacency matrix of one of the graphs $T_{2k}$ (for $k >2$), $C^{++}_{2k}$ (for $k >1$), $C^{+-}_{2k}$ (for $k >1$), $S_1$, $S_2$, $S_7$, $S_8$, $S_8^\prime$, $S_{14}$, and $S_{16}$ in Figures~\ref{fig:maxcycs1}, \ref{fig:maxcycs3}, \ref{fig:maxcycs5}, \ref{fig:maxcycs6}, \ref{fig:maxcycs8}, and \ref{fig:maxcycs10}.
	
	Moreover, every indecomposable cyclotomic $\Z$-matrix is contained in a maximal one.
\end{theorem}

\begin{theorem}[Cyclotomic matrices over $\Z[\sqrt{2}{]}$]
	\label{thm:class}
	Let $A$ be a maximal indecomposable cyclotomic matrix over the ring $\Z[\sqrt{2}]$ that is not a $\Z$-matrix.
	Then $A$ is equivalent to an adjacency matrix of one of the graphs $C_{2k}$ (for $k >1$), $C_{2k+1}$ (for $k >0$), $S_2^\ddag$, $S_4^{(1,\sqrt{2})}$, $S_4^{(2,\sqrt{2})}$, $S_4^{(3,\sqrt{2})}$, and $S_8^\dag$ in Figures~\ref{fig:maxcycs2}, \ref{fig:maxcycs4}, \ref{fig:maxcycs8}, and \ref{fig:maxcycs10}.
	
	Moreover, every indecomposable cyclotomic $\Z[\sqrt{2}]$-matrix is contained in a maximal one.
\end{theorem}

Let $\varphi$ denote the golden ratio, $1/2+\sqrt{5}/2$, so that $\Z[\varphi]$ is the ring of integers of $\Q(\sqrt{5})$.
We let $\bar{\varphi}$ denote the conjugate of the golden ratio, $1/2-\sqrt{5}/2$.

\begin{theorem}[Cyclotomic matrices over $\Z[\varphi{]}$]
	\label{thm:c5}
	Let $A$ be a maximal indecomposable cyclotomic matrix over the ring $\Z[\varphi]$ that is not a $\Z$-matrix.
	Then $A$ is equivalent to an adjacency matrix of one of the graphs $S_3$, $S_4^{(1,\varphi)}$, $S_4^{(2,\varphi)}$, $S_4^{(3,\varphi)}$, $S_6$, $S_8^{\dag \dag}$, and $S_8^\ddag$ in Figures~\ref{fig:maxcycs8} and \ref{fig:maxcycs10}.
	
	Moreover, every indecomposable cyclotomic $\Z[\varphi]$-matrix is contained in a maximal one.
\end{theorem}

\begin{theorem}[Cyclotomic matrices over $\Z[\sqrt{3}{]}$]
	\label{thm:c3}
	Let $A$ be a maximal indecomposable cyclotomic matrix over the ring $\Z[\sqrt{3}]$ that is not a $\Z$-matrix.
	Then $A$ is equivalent to an adjacency matrix of one of the graphs $S_2^\prime$, $S_2^\dag$, and $S_4^{(\sqrt{3})}$ in Figure~\ref{fig:maxcycs8}.
	
	Moreover, every indecomposable cyclotomic $\Z[\sqrt{3}]$-matrix is contained in a maximal one.
\end{theorem}

Theorem~\ref{thm:c5} can be proved by computation of $\Z[\varphi]$-matrices up to degree $8$ and for Theorem~\ref{thm:c3} it suffices to compute $\Z[\sqrt{3}]$-matrices up to degree $4$. 
By interlacing, for all $k \geqslant 2$, each matrix in $\mathfrak S_k^\prime$ contains at least one matrix from $\mathfrak S^\prime_{k-1}$.
From our computations, we have that there are no $\Z[\varphi]$-matrices in $\mathfrak S^\prime_9$, and hence, by interlacing, neither are there $\Z[\varphi]$-matrices in $\mathfrak S^\prime_k$ for $k > 9$ and similarly, there are no $\Z[\sqrt{3}]$-matrices in $\mathfrak S^\prime_k$ for $k > 4$.
Theorem~\ref{thm:class} follows from the technique in the author's paper \cite{Greaves:CycloEG11}, in particular, the proof technique strongly resembles Section~7 of that paper.
See the author's thesis for full details.

Let $\mathcal R$ be the compositum of all real quadratic integer rings $\mathcal O_{\Q(\sqrt{d})}$ where $d > 1$ is squarefree. 

\begin{theorem}[Cyclotomic matrices over $\mathcal R$]
	\label{thm:classR}
	Let $A$ be an indecomposable cyclotomic matrix over the ring $\mathcal R$.
	Then $A$ is a symmetric matrix over $\Z$, $\Z[\sqrt{2}]$, $\Z[\varphi]$, or $\Z[\sqrt{3}]$.
\end{theorem}

\begin{corollary}\label{cor:SnSdashn}
	For $n > 6$ we have $\mathfrak S_n = \mathfrak S^\prime_n$.
\end{corollary}

In Section~\ref{sec:pf}, after stating the Perron-Frobenius theorem, we prove Theorem~\ref{thm:classR} and Corollary~\ref{cor:SnSdashn}.

\section{Applying Perron-Frobenius theory}
\label{sec:pf}
As opposed to imaginary quadratic integer rings, by working with Hermitian matrices over real quadratic integer rings we do not possess the nice property of having a guaranteed integral characteristic polynomial, but we \emph{are} able to make use of the Perron-Frobenius Theorem which we state below.
\subsection{The Perron-Frobenius Theorem}

The \textbf{spectral radius} $\rho(A)$ of a square matrix $A$ is the maximum of the moduli of its eigenvalues.
We define the spectral radius $\rho(G)$ of the graph $G$ corresponding to $A$ to be the spectral radius of $A$.
A real matrix is called \textbf{nonnegative} if all its entries are nonnegative and a graph is called nonnegative if it has a nonnegative adjacency matrix.
Let $A$ and $B$ be real symmetric matrices of dimension $n$ and $m$ respectively with $n \geqslant m$.
We write $A \geqslant B$ if $A$ contains a principal submatrix such that $A - B$ is nonnegative; the inequality is strict unless $A = B$.
For the graphs $G$ and $H$ corresponding to $A$ and $B$ respectively, we write $G \geqslant H$.

\begin{theorem}[Perron-Frobenius Theorem]\cite[Theorem 8.8.1]{Godsil:AlgGraph}\label{thm:PerronFrobenius}
	Suppose $A$ is an indecomposable nonnegative $n \times n$ matrix. Then:
	\renewcommand{\labelenumi}{(\alph{enumi})}
	\begin{enumerate}
		\item The spectral radius $\rho = \rho(A)$ is a simple eigenvalue of $A$ and an eigenvector $\mathbf x$ is an eigenvector for $\rho$ if and only if no entries of $\mathbf x$ are zero, and all have the same sign.
		\item Suppose $A^\prime$ is a nonnegative $n \times n$ matrix such that $A - A^\prime$ is nonnegative. Then $\rho(A^\prime) \leqslant \rho(A)$ with equality if and only if $A = A^\prime$.
	\end{enumerate}
\end{theorem}

\emph{Remark.} Suppose $G$ is a connected graph and $H$ is a nonnegative graph.
An implication of Perron-Frobenius together with interlacing is that if $G > H$ then $\rho(G) > \rho(H)$.
The nonnegative graphs $P^{(1)}_n$ (for $n \geqslant 3$), $P^{(2)}_n$ (for $n \geqslant 2$), $P^{(3)}_n$ (for $n \geqslant 2$), and $Q_n$ (for $n \geqslant 3$) in Figure~\ref{fig:cycloPaths} have an eigenvalue of $2$ corresponding to an eigenvector given by the numbers beneath their vertices.
By Theorem~\ref{thm:PerronFrobenius}, since the eigenvectors given are positive, the graphs $P^{(1)}_n$ (for $n \geqslant 3$), $P^{(2)}_n$ (for $n \geqslant 2$), $P^{(3)}_n$ (for $n \geqslant 2$), and $Q_n$ (for $n \geqslant 3$) all have spectral radius $2$.

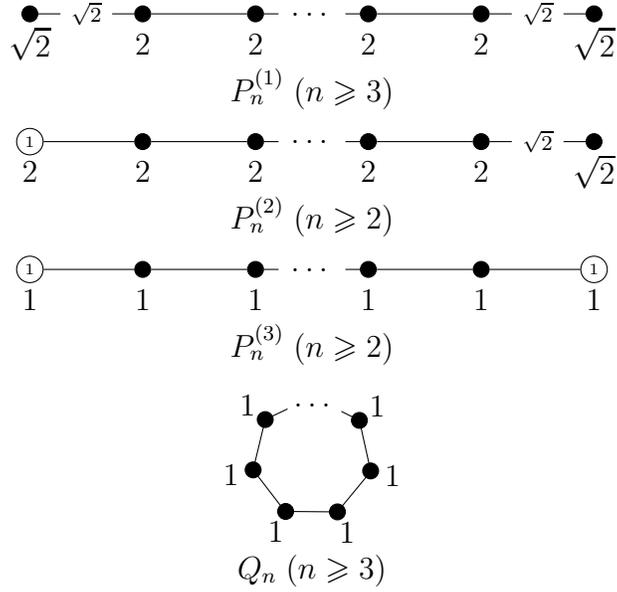
\begin{figure}[htbp]
	\centering
		\begin{tikzpicture}[xscale=1.5]
			\begin{scope}
				\foreach \pos/\name in {{(0,0)/a}, {(1,0)/b}, {(2,0)/c}, {(3,0)/d}, {(4,0)/e}, {(5,0)/f}}
					\node[vertex] (\name) at \pos {}; 
				\foreach \pos/\name in {{(0,-0.4)/\sqrt{2}}, {(1,-0.4)/2}, {(2,-0.4)/2}, {(3,-0.4)/2}, {(4,-0.4)/2}, {(5,-0.4)/\sqrt{2}}}
					\node at \pos {$\name$};
				\node (s1) at (2.3,0) {};
				\node (s2) at (2.7,0) {};
				\node at (2.5,0) {$\dots$};
				\foreach \edgetype/\source/ \dest /\weight in {pedge/b/c/{}, pedge/d/e/{}, pedge/c/s1/{}, pedge/s2/d/{} }
					\path[\edgetype] (\source) -- node[weight2] {$\weight$} (\dest);
				\path[pedge] (a) -- node[weight] {$\sqrt{2}$} (b);
				\path[pedge] (e) -- node[weight] {$\sqrt{2}$} (f);
				\node at (2.5,-1) {$P^{(1)}_n \; (n \geqslant 3)$};
			\end{scope}
			\begin{scope}[yshift=-1.7cm]		
					\foreach \pos/\name in {{(1,0)/b}, {(2,0)/c}, {(3,0)/d}, {(4,0)/e}, {(5,0)/f}}
						\node[vertex] (\name) at \pos {}; 
					\foreach \type/\pos/\name in {{pc/(0,0)/a}}
						\node[\type] (\name) at \pos {$1$};
					\foreach \pos/\name in {{(0,-0.4)/2}, {(1,-0.4)/2}, {(2,-0.4)/2}, {(3,-0.4)/2}, {(4,-0.4)/2}, {(5,-0.4)/\sqrt{2}}}
						\node at \pos {$\name$};
					\node (s1) at (2.3,0) {};
					\node (s2) at (2.7,0) {};
					\node at (2.5,0) {$\dots$};
					\foreach \edgetype/\source/ \dest /\weight in {pedge/a/b/{}, pedge/b/c/{}, pedge/d/e/{}, pedge/c/s1/{}, pedge/s2/d/{} }
						\path[\edgetype] (\source) -- node[weight2] {$\weight$} (\dest);
					\path[pedge] (e) -- node[weight] {$\sqrt{2}$} (f);
					\node at (2.5,-1) {$P^{(2)}_n \; (n \geqslant 2)$};
			\end{scope}
			\begin{scope}[yshift=-3.4cm]
				\foreach \pos/\name in {{(1,0)/b}, {(2,0)/c}, {(3,0)/d}, {(4,0)/e}}
					\node[vertex] (\name) at \pos {}; 
				\foreach \type/\pos/\name in {{pc/(0,0)/a}, {pc/(5,0)/f}}
					\node[\type] (\name) at \pos {$1$};
				\foreach \pos/\name in {{(0,-0.4)/1}, {(1,-0.4)/1}, {(2,-0.4)/1}, {(3,-0.4)/1}, {(4,-0.4)/1}, {(5,-0.4)/1}}
					\node at \pos {$\name$};
				\node (s1) at (2.3,0) {};
				\node (s2) at (2.7,0) {};
				\node at (2.5,0) {$\dots$};
				\foreach \edgetype/\source/ \dest /\weight in {pedge/a/b/{}, pedge/b/c/{}, pedge/d/e/{}, pedge/e/f/{}, pedge/c/s1/{}, pedge/s2/d/{} }
					\path[\edgetype] (\source) -- node[weight2] {$\weight$} (\dest);
				\node at (2.5,-1) {$P^{(3)}_n \; (n \geqslant 2)$};
			\end{scope}
			\end{tikzpicture}
			
			\begin{tikzpicture}
					\newdimen\rad
					\rad=0.8cm
					\newdimen\lorad
					\lorad=0.7cm
					\newdimen\radi
					\radi=1.1cm

					\foreach \x in {141,192,244,295,347,398}
					{
						\draw (\x:\radi) node[empty] {$1$};
				    	\draw (\x:\rad) node[vertex] {};
						\draw[pedge] (\x:\rad) -- (\x+51:\rad);
				    }
					\draw[pedge] (90:\rad) -- (141:\rad);
					\draw (87:\lorad) node[empty] {$\dots$};
					\node at (0,-1.5) {$Q_n \; (n \geqslant 3)$};
		\end{tikzpicture}
	\caption{Four infinite families of nonnegative cyclotomic $\Z[\sqrt{2}]$-graphs each having spectral radius $2$. The numbers on the vertices correspond to an eigenvector with largest eigenvalue $2$. The subscript is the number of vertices.}
	\label{fig:cycloPaths}
\end{figure}

\subsection{Cyclotomic matrices over the compositum of real quadratic integer rings}
\label{sec:classR}

In this section we prove that all matrices in $\mathfrak S^\prime_n$ are necessarily symmetric matrices over one of the rings $\Z$, $\Z[\sqrt{2}]$, $\Z[\varphi]$, or $\Z[\sqrt{3}]$.
Set $R = \mathcal R$, the compositum of all real quadratic integer rings, and let $K$ be the Galois closure of the field generated by elements of $R$ over $\Q$.
Let $A$ be an $R$-matrix in $\mathfrak S^\prime_n$ and let $G$ be its corresponding $R$-graph.
By Lemma~\ref{lem:maxDeg4}, we need only consider entries of $A$ from the set $R^\prime = \{0, \pm 1, \pm \sqrt{2}, \pm \varphi, \pm \bar \varphi, \pm \sqrt{3}, \pm 2\}$; these are the only real algebraic integers from $R$ whose conjugates all square to at most $4$.
Otherwise, we can apply some automorphism $\sigma \in \operatorname{Gal}(K/\Q)$ to $A$ so that some entry squares to more than $4$.
Therefore, without loss of generality, we can take $R$ to be the ring generated by $R^\prime$ over $\Z$ and we can set $K = \Q(\sqrt{2}, \sqrt{3}, \sqrt{5})$.

First we deal with the possibility of $G$ containing a subgraph equivalent to the following graphs:
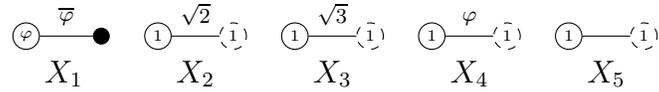
\[
	\begin{tikzpicture}[auto]
		\foreach \pos/\name in {{(1,0)/b}}
			\node[vertex] (\name) at \pos {}; 
		\foreach \type/\pos/\name in {{pc/(0,0)/a}}
			\node[\type] (\name) at \pos {$\varphi$};
		\foreach \edgetype/\source/ \dest /\weight in {pedge/a/b/{\bar \varphi}}
			\path[\edgetype] (\source) -- node[weight2] {$\weight$} (\dest);
		\node at (0.5,-0.5) {$X_1$};
	\end{tikzpicture}
	\begin{tikzpicture}[auto]
		\foreach \type/\pos/\name in {{pc/(0,0)/a},{nc/(1,0)/b}}
			\node[\type] (\name) at \pos {$1$};
		\foreach \edgetype/\source/ \dest /\weight in {pedge/a/b/{\sqrt{2}}}
			\path[\edgetype] (\source) -- node[weight2] {$\weight$} (\dest);
		\node at (-0.5,0) {};
		\node at (0.5,-0.5) {$X_2$};
	\end{tikzpicture}
	\begin{tikzpicture}[auto]
		\foreach \type/\pos/\name in {{pc/(0,0)/a},{nc/(1,0)/b}}
			\node[\type] (\name) at \pos {$1$};
		\foreach \edgetype/\source/ \dest /\weight in {pedge/a/b/{\sqrt{3}}}
			\path[\edgetype] (\source) -- node[weight2] {$\weight$} (\dest);
		\node at (-0.5,0) {};
		\node at (0.5,-0.5) {$X_3$};
	\end{tikzpicture}
	\begin{tikzpicture}[auto]
		\foreach \type/\pos/\name in {{pc/(0,0)/a},{nc/(1,0)/b}}
			\node[\type] (\name) at \pos {$1$};
		\foreach \edgetype/\source/ \dest /\weight in {pedge/a/b/{\varphi}}
			\path[\edgetype] (\source) -- node[weight2] {$\weight$} (\dest);
		\node at (-0.5,0) {};
		\node at (0.5,-0.5) {$X_4$};
	\end{tikzpicture}
	\begin{tikzpicture}[auto]
		\foreach \type/\pos/\name in {{pc/(0,0)/a},{nc/(1,0)/b}}
			\node[\type] (\name) at \pos {$1$};
		\foreach \edgetype/\source/ \dest /\weight in {pedge/a/b/{}}
			\path[\edgetype] (\source) -- node[weight2] {$\weight$} (\dest);
		\node at (-0.5,0) {};
		\node at (0.5,-0.5) {$X_5$};
	\end{tikzpicture}
\]
We have exhaustively checked all $R$-supergraphs of $X_1$, $X_2$, $X_3$, $X_4$, and $X_5$ that are in $\mathfrak S^\prime_n$ for each $n \in \N$. 
These supergraphs are all equivalent to subgraphs of either $S_2^\dag$, $C_4^{+-}$, $S_4^{(1,\varphi)}$, $S_4^{(3,\varphi)}$, $S_4^{(1,\sqrt{2})}$, $S_7$, $S_8$, or $S_8^\prime$ (see Figures~\ref{fig:maxcycs3}, \ref{fig:maxcycs8}, and \ref{fig:maxcycs10}) and hence are all either $\Z[\sqrt{2}]$-graphs, $\Z[\sqrt{3}]$-graphs, or $\Z[\varphi]$-graphs.
This computation can be checked with little effort; we used {\tt PARI/GP}~\cite{PARI} to implement the following algorithm.
Start with a seed graph $X$ (one of $X_1$, $X_2$, $X_3$, $X_4$, and $X_5$).
Consider all possible ways of attaching a vertex to $X$ so that the matrix corresponding to the resulting graph is in $\mathfrak S_3^\prime$.
For each of the resulting graphs, on (say) $k$ vertices, repeat the process so that the matrices corresponding to the resulting graphs are in $\mathfrak S_{k+1}^\prime$.
This process terminates, and we obtain the list of all $R$-supergraphs of the $X_j$ that are in $\mathfrak S^\prime_n$ for each $n \in \N$.
Henceforth we assume that $X_1$, $X_2$, $X_3$, $X_4$, and $X_5$ are not equivalent to any subgraph of $G$.
We can also exclude $\pm 2$ from being an entry of our matrix $A$ since, by Lemma~\ref{lem:maxDeg4}, any connected graph strictly containing either $S_1$ or $S_2$ does not correspond to a matrix in $\mathfrak S^\prime_n$ for any $n$.

Let $A^\prime$ be a smallest principal submatrix of $A$ with respect to having at least two irrational entries $\alpha$ and $\beta$ such that its corresponding $R$-graph $G^\prime$ is connected.
Suppose $\alpha$ is not conjugate to $\pm \beta$, i.e., $\alpha$ and $\pm \beta$ do not have the same minimal polynomial.
We will show that this supposition violates the condition that $A$ is in $\mathfrak S^\prime_n$.
We can assume that at least one of $\alpha$ and $\beta$ (say $\alpha$) is not equal to $\pm \sqrt{2}$.
Observe that, by a combination of switching and Galois conjugation (using automorphisms from $\operatorname{Gal}(K/\Q)$), we can make all the edge-weights of $G^\prime$ positive and hence we assume that all the off-diagonal entries of $A^\prime$ are nonnegative.

If $G^\prime$ is a triangle then, since we have excluded the subgraphs $X_1$, $X_2$, $X_3$, $X_4$, and $X_5$, we can find a graph $H$ equivalent to $G^\prime$ that satisfies $H > Q_3$.
By the Perron-Frobenius Theorem, the spectral radius of $H$ is strictly greater than $2$; hence, by interlacing, $A$ is equivalent to a matrix that is not in $\mathfrak S_n^\prime$.
Therefore $A$ is not in $\mathfrak S_n^\prime$.
Otherwise, if $G^\prime$ is not a triangle then $G^\prime$ must be a path.
Since $A^\prime$ is minimal with respect to the condition of containing both $\alpha$ and $\beta$ as entries, any induced subpath $p_1p_2 \dots p_k$ of $G^\prime$ must have $w(p_i,p_{i+1}) = 1$ when $i$ is equal to neither $1$ nor $k-1$.
Moreover, the minimality also implies that the charge of $p_j$ for $j \in \left \{2,\dots,k-1\right \}$ is either $0$ or $\pm 1$.

We consider two cases for $G^\prime$: the case where $G^\prime$ is uncharged and the case where $G^\prime$ has a charge.
In the first case we have $H > P^{(1)}_n$ for some $n$ and in the second, we have either $H > P^{(2)}_n$ or $H > P^{(3)}_n$ for some $n$ where, in each case, $H$ is an $R$-graph that is equivalent to $G^\prime$.
By the Perron-Frobenius Theorem, the spectral radius of $H$ is strictly greater than $2$ and hence, by interlacing, $A$ is not in $\mathfrak S_n^\prime$.
Therefore, we have established the following result.

\begin{proposition}\label{pro:oneAtaTime}
	Let $A$ be an indecomposable $\mathcal R$-matrix having as entries two irrational integers $\alpha$ and $\beta$ with $\alpha$ not conjugate to $\pm \beta$.
	Then $A$ is not in $\mathfrak S^\prime_n$.
\end{proposition}

Theorem~\ref{thm:classR} follows immediately.

\subsection{Elements of $\mathfrak S^\prime_n \backslash \mathfrak S_n$}
\label{sec:equalsets}

Here we give a proof of Corollary~\ref{cor:SnSdashn} and enumerate all elements in $\mathfrak S^\prime_n \backslash \mathfrak S_n$ for $n \leqslant 6$.
In Table~\ref{tab:sndsn}, we have tabulated the number of elements of the set $\mathfrak S^\prime_n \backslash \mathfrak S_n$ for $n \leqslant 6$, these are given working up to equivalence.
With respect to Theorem~\ref{thm:classR}, we have also recorded the number of elements in 
$\mathfrak S^\prime_n \backslash \mathfrak S_n$ that lie in each $\mathcal O_{\Q(\sqrt{d})}$-matrix ring for $d > 1$.
We remark that each element of $\mathfrak S^\prime_n$ is contained in a maximal cyclotomic matrix.
Since all subgraphs of the infinite families of maximal cyclotomic matrices are in $\mathfrak S_n$, one can find elements $\mathfrak S^\prime_n \backslash \mathfrak S_n$ by checking subgraphs of the sporadic maximal cyclotomic matrices.

\begin{table}[htbp]
	\begin{center}
	\begin{tabular}{|c|c|c|c|c|}
		\hline
		$n$ & $|\mathfrak S^\prime_n \backslash \mathfrak S_n|$ & $\Z[\varphi]$ & $\Z[\sqrt{2}]$ & $\Z[\sqrt{3}]$ \\
		\hline
		$1$ & $3$ & $1$ & $1$ & $1$ \\
		$2$ & $7$ & $6$ & $1$ & $0$ \\
		$3$ & $4$ & $3$ & $1$ & $0$ \\
		$4$ & $6$ & $6$ & $0$ & $0$ \\
		$5$ & $4$ & $4$ & $0$ & $0$ \\
		$6$ & $1$ & $1$ & $0$ & $0$ \\
		\hline
	\end{tabular}
	\end{center}
	\caption{Up to equivalence, the number of elements of the set $\mathfrak S^\prime_n \backslash \mathfrak S_n$ for $n \leqslant 6$.}
	\label{tab:sndsn}
\end{table}

Now we give a lemma resembling the crystallographic criterion for a Coxeter graph, see Humphreys~\cite[Proposition~6.6]{Hump:Coxeter90}. 

\begin{lemma}\label{lem:snGinvariant}
	Let $A \in \mathfrak S^\prime_n$ be a $\Z[\sqrt{2}]$-matrix having all its charges in $\Z$ and let $G$ be its associated graph.
	Then every cycle of $G$ has an even number of edges of weight $\pm \sqrt{2}$.
	Hence $A$ is in $\mathfrak S_n$.
\end{lemma}

\begin{proof}
	Let $\sigma$ be the nontrivial automorphism of $\Z[\sqrt{2}]$ which sends $\sqrt{2}$ to $-\sqrt{2}$.
	Suppose for a contradiction that $G$ contains a cycle having an odd number of edges with weight $\pm \sqrt{2}$ and let $C$ be a smallest such cycle.
	\paragraph{Case 1} 
	\label{par:case_1}
	$C$ is uncharged.
	In this case we can switch either $C$ or $\sigma(C)$ in such a way that the resulting nonnegative cycle $C^\prime$ has $C^\prime > Q_k$ for some $k$.
	Hence, $\rho(C^\prime) > \rho(Q_k) = 2$ and so, by interlacing, we have $\rho(A) \geqslant \rho(C^\prime) > \rho(Q_k) = 2$.
	\paragraph{Case 2} 
	\label{par:case_22}
	$C$ is charged.
	As in the previous section we can exclude $X_2$ and $X_5$ as subgraphs of $G$.
	In the case when $C$ is a triangle, one can find an equivalent cycle $C^\prime$ satisfying $C^\prime > Q_3$.
	Otherwise, $C$ contains a subpath equivalent to a path $C^\prime$ where either $C^\prime > P_k^{(2)}$ or $C^\prime > P_k^{(3)}$ for some $k$.
	Therefore, in each case, $A \not \in \mathfrak S^\prime_n$ which is a contradiction.

	On the other hand, it can be readily seen that if all the cycles of $G$ have an even number of edges of weight $\pm \sqrt{2}$, then $G$ is Galois invariant. 	
\end{proof}

Finally, we give a proof of Corollary~\ref{cor:SnSdashn}.

\begin{proof}[Proof of Corollary~\ref{cor:SnSdashn}]
	We have computed all the sets $\mathfrak S^\prime_n$ and $\mathfrak S_n$ for $n \leqslant 8$.
	We have that $\mathfrak S^\prime_7 = \mathfrak S_7$ and $\mathfrak S^\prime_8 = \mathfrak S_8$.
	By computation and Proposition~\ref{pro:oneAtaTime}, we know that all matrices in $\mathfrak S^\prime_n$ for $n > 8$ are $\Z[\sqrt{2}]$-matrices.
	Thus, it suffices to consider only $\Z[\sqrt{2}]$-matrices.
	From our computation we know that all $\Z[\sqrt{2}]$-matrices in $\mathfrak S^\prime_5$ have all their charges in $\Z$, hence, by interlacing, the same must be true for the sets $\mathfrak S^\prime_k$ for all $k>5$.
	The result then follows from Lemma~\ref{lem:snGinvariant}.
\end{proof}

\bibliographystyle{amsalpha}
\bibliography{../bib}
\end{document}

%% file: packages.tex
\usepackage{amsthm, amssymb, amsmath}
\usepackage{mathrsfs, latexsym, color,fullpage, enumerate}
\usepackage[pdftex]{}

\numberwithin{equation}{section}

\theoremstyle{plain}
\newtheorem{theorem}{Theorem}[section]
\newtheorem{proposition}[theorem]{Proposition}
\newtheorem{lemma}[theorem]{Lemma}
\newtheorem{corollary}[theorem]{Corollary}

\theoremstyle{definition}

\theoremstyle{remark}

\renewcommand{\bar}{\overline}

\newcommand{\abs}[1]{\lvert#1\rvert}

\newcommand{\N}{\mathbb{N}}
\newcommand{\Z}{\mathbb{Z}}
\newcommand{\Q}{\mathbb{Q}}

\usepackage{pgf}
\usepackage{tikz}
\usetikzlibrary{arrows,decorations.markings,shapes}
\tikzstyle{vertex}=[circle, draw, fill=black, inner sep=0pt, minimum width=6pt]
\tikzstyle{vert}=[circle, draw=black, fill=white, inner sep=0pt, minimum width=6pt]
\tikzstyle{zc}=[circle, draw, fill=black, inner sep=0pt, minimum width=6pt]
\tikzstyle{pc}=[circle, draw=black, inner sep=1pt, minimum width=10pt, font=\tiny] 
\tikzstyle{nc}=[circle, draw=black, style=dashed, inner sep=1pt, minimum width=10pt, font=\tiny] 
\tikzstyle{pedge}=[draw,-]
\tikzstyle{wwedge}=[draw,-,double,double distance = 2pt,]
\tikzstyle{wwwedge}=[draw,-,postaction={decorate}, decoration={markings,mark = at position 0.65 with {\arrow{stealth} },mark = at position 0.45 with {\arrow{stealth} }, mark = at position 0.55 with {\arrow{stealth} } }]
\tikzstyle{wwwwedge}=[draw,-,postaction={decorate}, decoration={markings,mark = at position 0.7 with {\arrow{stealth} },mark = at position 0.4 with {\arrow{stealth} }, mark = at position 0.5 with {\arrow{stealth} }, mark = at position 0.6 with {\arrow{stealth} } }]
\tikzstyle{wwwnedge}=[draw,densely dashed,postaction={decorate}, decoration={markings,mark = at position 0.65 with {\arrow{stealth} },mark = at position 0.45 with {\arrow{stealth} }, mark = at position 0.55 with {\arrow{stealth} } }]
\tikzstyle{wwnedge}=[draw=black,densely dashed,double=white,double distance = 2pt]
\tikzstyle{nedge}=[draw,densely dashed]
\tikzstyle{weight}= [draw=white, fill=white, font=\scriptsize]
\tikzstyle{weight2}= [font=\scriptsize]
\tikzstyle{empty}=[circle, draw=white, inner sep=2pt, fill=white, minimum width=4pt]
\tikzstyle{ghost}=[circle, draw=black, inner sep=2pt, style=densely dashed, minimum width=8pt, font=\tiny]
\tikzstyle{ghostc}=[circle, draw=black, inner sep=1pt, style=densely dashed, minimum width=10pt, font=\tiny]

%% file: master.bbl
\providecommand{\bysame}{\leavevmode\hbox to3em{\hrulefill}\thinspace}
\providecommand{\MR}{\relax\ifhmode\unskip\space\fi MR }
\providecommand{\MRhref}[2]{%
  \href{http://www.ams.org/mathscinet-getitem?mr=#1}{#2}
}
\providecommand{\href}[2]{#2}
\begin{thebibliography}{CGSS76}

\bibitem[Bou02]{Bour:LieGroups}
N.~Bourbaki, \emph{Lie groups and {L}ie algebras chapters 4-6}, Springer, 2002.

\bibitem[Cau29]{Cau:Interlace}
Augustin-L. Cauchy, \emph{Sur l'{\'e}quation {\`a} l'aide de laquelle on
  d{\'e}termine les in{\'e}galit{\'e}s s{\'e}culaires des mouvements des
  plan\`{e}tes}, Oeuvres compl\`{e}tes, IIi{\`e}me S{\'e}rie, Gauthier-Villars,
  1829.

\bibitem[CGSS76]{Cam:LineSystems76}
P.~J. Cameron, J.-M. Goethals, J.~J. Seidel, and E.~E. Shult, \emph{Line
  graphs, root systems, and elliptic geometry}, J. Algebra \textbf{43} (1976),
  no.~1, 305--327.

\bibitem[CR90]{Cvet:GraphSurvey90}
D.~Cvetkovi{\'c} and P.~Rowlinson, \emph{The largest eigenvalue of a graph: A
  survey}, Linear and Multilinear Algebra \textbf{28} (1990), 3--33.

\bibitem[CST94]{Seid:Signed94}
P.J. Cameron, J.J. Seidel, and S.V. Tsaranov, \emph{Signed graphs, root
  lattices, and {C}oxeter groups}, Journal of Algebra \textbf{164} (1994),
  173--209.

\bibitem[Fis05]{Fisk:Interlace05}
Steve Fisk, \emph{A very short proof of {C}auchy's interlace theorem for
  eigenvalues of {H}ermitian matrices}, Amer. Math. Monthly \textbf{112}
  (2005), no.~2, 118.

\bibitem[GR00]{Godsil:AlgGraph}
Chris Godsil and Gordon Royle, \emph{Algebraic graph theory}, Graduate Texts in
  Mathematics, New York: Springer, 2000.

\bibitem[Gre]{Greaves:CycloEG11}
Gary Greaves, \emph{Cyclotomic matrices over the {E}isenstein and {G}aussian
  integers}, Jounal of Algebra (to appear).

\bibitem[Hum90]{Hump:Coxeter90}
James~E. Humphreys, \emph{Reflection groups and {C}oxeter groups}, Cambridge
  University Press, 1990.

\bibitem[Hwa04]{Hwang:Interlace04}
Suk-Geun Hwang, \emph{Cauchy's interlace theorem for eigenvalues of {H}ermitian
  matrices}, Amer. Math. Monthly \textbf{111} (2004), 157--159.

\bibitem[Kro57]{Kron:cyclo57}
Leopold Kronecker, \emph{Zwei {S}{\"a}tze {\"u}ber {G}leichungen mit
  ganzzahligen {C}oefficienten}, J. Reine Angew. Math. \textbf{53} (1857),
  173--175.

\bibitem[Leh33]{Lehmer:33Cyclo}
Derrick~H. Lehmer, \emph{Factorization of certain cyclotomic functions}, Annals
  of mathematics \textbf{34} (1933), no.~3, 461--479.

\bibitem[Mah62]{Mahler:Measure1962}
Kurt Mahler, \emph{On some inequalities for polynomials in several variables},
  J. London Math. Soc. \textbf{37} (1962), 341--344.

\bibitem[MS07]{McKee:IntSymCyc07}
James McKee and Chris Smyth, \emph{Integer symmetric matrices having all their
  eigenvalues in the interval [-2,2]}, Journal of Algebra \textbf{317} (2007),
  260--290.

\bibitem[MS12]{McKee:noncycISM09}
\bysame, \emph{Integer symmetric matrices of small spectral radius and small
  {M}ahler measure}, International Mathematics Research Notices (2012), no.~1,
  102--136.

\bibitem[PAR11]{PARI}
PARI/GP, \emph{version 2.3.4}, \texttt{http://pari.math.u-bordeaux.fr}, 2011.

\bibitem[Smi70]{Smith:CycloG}
John~H. Smith, \emph{Some properties of the spectrum of a graph}, Combinatorial
  {S}tructures and their {A}pplications ({P}roc. {C}algary {I}nternat. {C}onf.,
  {C}algary, {A}lta., 1969), Gordon and Breach, New York, 1970, pp.~403--406.

\bibitem[Smy08]{Smyth:MMsurvey08}
Chris Smyth, \emph{The {M}ahler measure of algebraic numbers: a survey}, Number
  Theory and Polynomials (James McKee and Chris Smyth, eds.), London
  Mathematical Society Lecture Note Series, Cambridge University Press, 2008.

\bibitem[Tay11]{GTay:cyclos10}
Graeme Taylor, \emph{Cyclotomic matrices and graphs over the ring of integers
  of some imaginary quadratic fields}, Journal of Algebra \textbf{331} (2011),
  523--545.

\bibitem[Tay12]{GTay:Lehmer11}
\bysame, \emph{Lehmer's conjecture for matrices over the ring of integers of
  some imaginary quadratic fields}, Journal of Number Theory \textbf{132}
  (2012), no.~4, 590--607.

\end{thebibliography}
